\documentclass[a4]{amsart}
\usepackage{amssymb}
\usepackage{amscd}
\usepackage{verbatim,ifthen}
\usepackage{graphicx}
\usepackage{xypic}
\usepackage{cite}
\usepackage{amsthm}

\usepackage{latexsym}

\newtheorem*{lemma*}{Lemma}

\newtheorem{thm}{Theorem}

\newtheorem{lemma}[thm]{Lemma}

\theoremstyle{definition}
\newtheorem{defn}[thm]{Definition}

\usepackage{hyperref}

\begin{document}

\title[Optimal Holomorphic Functional Calculus for the Ornstein-Uhlenbeck operator]{Optimal angle of the holomorphic functional calculus for the Ornstein-Uhlenbeck operator} 

\author{Sean Harris}
\address{Hanna Neumann Building \#145, Science Road
The Australian National University
Canberra ACT 2601.}
\email{Sean.Harris@anu.edu.au}     

\date{\today}

\thanks{The author gratefully acknowledges financial support by the discovery Grant DP160100941 of the Australian Research Council. This research is also supported by an Australian Government Research Training Program (RTP) Scholarship.}

 \begin{abstract}
We give a simple proof of the fact that the classical Ornstein-Uhlenbeck operator $L$  is R-sectorial of angle $\arcsin|1-2/p|$ on $L^{p}(\mathbb{R}^{d},\mu)$ for $1<p<\infty$, where $\mu$ is the standard Gaussian measure with density $d\mu = (2\pi)^{-\frac{d}{2}}\exp(-|x|^2/2)dx$.
Applying the abstract holomorphic functional calculus theory of Kalton and Weis, this immediately gives a new proof of the fact that $L$ has a bounded $H^{\infty}$ functional calculus with this optimal angle.
\end{abstract}

 
 \maketitle

\section{Introduction}
The Ornstein-Uhlenbeck operator appears in many areas of mathematics: as the number operator of quantum field theory, the analogue of the Laplacian in the Malliavin calculus, the generator of the transition semigroup associated with the simplest mean-reverting stochastic process (the Ornstein-Uhlenbeck process), or as the operator associated with the classical Dirichlet form on $\mathbb{R}^d$ equipped with the Gaussian measure $d\mu = (2\pi)^{-\frac{d}{2}} e^{-|x|^2/2}dx$. For the sake of this paper, the Ornstein-Uhlenbeck operator will be defined via the Ornstein-Uhlenbeck semigroup $\{T_t\}_{t>0}$ whose action on $f \in L^p(\mu)$ is
\[T_tf(x) = \int_{\mathbb{R}^d} M_t(x,y)f(y)dy, \text{ for }x \in \mathbb{R}^d\]
where $M_t:\mathbb{R}^{2d} \to \mathbb{R}$ is given by 
\begin{equation}\label{defnOU}
(x,y) \mapsto \frac{1}{(2\pi)^\frac{d}{2}}\left(\frac{1}{1-e^{-2t}}\right)^\frac{d}{2}\exp\left( -\frac{1}{2}\frac{|e^{-t}x - y|^2}{(1-e^{-2t})} \right),
\end{equation}
the Mehler kernel.

Let us recall the basic properties of the Ornstein-Uhlenbeck semigroup used in this article. For each $p \in [1, \infty]$ and each $t>0$, the map $f \mapsto T_tf$ is bounded $L^p(\mu) \to L^p(\mu)$, with operator norm at most $1$, and is a positive operator. For $p \in [1, \infty)$, $T_t:L^p(\mu) \to L^p(\mu)$ is a $C_0$ semigroup, i.e. as $t \to 0$, $T_t \to I$ strongly and $T_t T_s = T_{t+s}$ for all $t, s>0$. For a proof of these preliminary facts, see for example Theorem 2.5 of \cite{Urbina}. It should be noted that although the Ornstein-Uhlenbeck semigroup arises in many different areas of mathematics, these basic properties can be proven solely with use of the explicit kernel and elementary techniques. It is a simple calculation to show that $T_t$ is bounded with norm $1$ on both $L^\infty(\mu)$ and $L^1(\mu)$, from which interpolation can be used to deduce boundedness with norm $1$ on $L^p(\mu)$ for $p\in [1, \infty]$. Positivity follows from non-negativity of the Mehler kernel. Strong continuity of the semigroup follows as in typical proofs of the strong continuity of the classical heat semigroup, and the semigroup property follows from a somewhat tedious exercise in integrating Gaussian functions. It should be noted that by using other representations of the Ornstein-Uhlenbeck semigroup, such as a spectral multiplier for the multivariate Hermite ONB of $L^2(\mu)$ or through a different representation via an integral kernel, one may prove some of these results even more simply, however the difficulty then becomes showing that all these representations for the Ornstein-Uhlenbeck semigroup are equivalent (for example, see \cite{Nualart}).
We consider the generator of the Ornstein-Uhlenbeck semigroup on $L^p(\mu)$, $p \in [1, \infty)$, whose negative we shall call the Ornstein-Uhlenbeck operator and denote by $L$. This operator is a closed densely-defined unbounded operator on $L^p(\mu)$, $p \in [1, \infty)$, which uniquely determines $T_t$. Thus from here on, we will use the notation $\exp(-tL)$ for the operator $T_t$, on any of these spaces.

This paper presents a new proof of the following theorem.
\begin{thm}\label{thmGoodHinfty}
For $p \in (1, \infty)$, the Ornstein-Uhlenbeck operator has a bounded $H^\infty(\Sigma_{\theta_p})$ functional calculus on $L^p(\mu)$, where $\sin (\theta_p) = \left|1-\frac{2}{p}\right|$.
\end{thm}
See [5] for the theory of the $H^{\infty}$ functional calculus. That $L$ has a bounded $H^\infty$ functional calculus (of some angle $\theta<\pi$) follows from general results in the theory of the $H^\infty$ functional calculus (for example, Theorem 10.7.13 of \cite{AnalysisInBanachSpacesV2} states that any generator of an analytic semigroup on an $L^p$ space for $p\in (1, \infty)$ which is a positive contraction semigroup for real time has a bounded $H^\infty$ functional calculus of some angle less than $\frac{\pi}{2}$). The difficulty in Theorem \ref{thmGoodHinfty} is to prove the boundedness of the calculus with precisely the optimal angle $\theta_{p}$.

Theorem \ref{thmGoodHinfty} was originally proven by Garc\'{i}a-Cuerva, Mauceri, Meda, Sj\"{o}gren and Torrea in \cite{GMMST}, also proving that $\theta_p$ is optimal. They use Mauceri's abstract multiplier theorem to reduce the problem to precisely estimating $u\mapsto ||L^{iu}||$.  To do so, they express $L^{iu}$ as an integral of the semigroup, using a carefully chosen contour of integration. They then consider the kernels of operators corresponding to different parts of the contour, and decompose them into a local and global part. To treat the global parts they then use a range of subtle kernel estimates.

In \cite{CD}, Carbonaro and Dragi\v{c}evi\'{c} reproved and extended the result of Theorem \ref{thmGoodHinfty} to treat arbitrary generators of symmetric contraction semigroups on an $L^p$ space over a $\sigma$-finite measure space. Note that as they work on abstract $L^p$ spaces, their result gives dimension independent estimates working over $\mathbb{R}^d$. For their proof, they first reduce the problem to proving a bilinear embedding for the semigroup, with constants depending optimally on the angle $\theta_{p}$. They then use the Bellman function method, controlling the bilinear form by an optimally (depending on $p$) chosen function. This function turns out to be a known Bellman function introduced by Nazarov and Treil, but just proving that it has the right properties is a highly non-trivial task.

In contrast, the proof presented in this paper is based on the well-known result that in $L^p$ spaces, the optimal angle of the $H^\infty$ functional calculus of an operator is equal to its optimal angle of R-sectoriality (see \cite{AnalysisInBanachSpacesV2} for the theory of R-sectoriality, and its Theorem 10.7.13 for a proof of the stated result). Our proof that the latter is equal to $\theta_p$ uses Theorem 10.3.3 of \cite{AnalysisInBanachSpacesV2}, which states an equivalence between an operator $A$ being R-sectorial of angle $\theta<\frac{\pi}{2}$ and $-A$ being the generator of an analytic semigroup of angle $\frac{\pi}{2}-\theta$ which is R-bounded on each smaller sector. To deduce R-boundedness of the Ornstein-Uhlenbeck semigroup on such sectors, a standard result on R-boundedness of integral operators with radially decaying kernels is employed (Proposition 8.2.3 of \cite{AnalysisInBanachSpacesV2}). This key step only requires simple manipulations of the kernel for the Ornstein-Uhlenbeck semigroup. It is based on an approach designed by van Neerven and Portal in \cite{vNPWeylGaussian}, where they recover classical results about the Ornstein-Uhlenbeck semigroup in a very direct manner. Their idea is to separate algebraic difficulties from analytic difficulties by considering a non-commutative functional calculus of the Gaussian position and momentum operators (the Weyl calculus). Using this calculus, one sees how to modify the kernels in a way that makes their analysis straightforward. A posteriori, the use of the Weyl calculus can be removed, and the proof can be read as a simple computation exploiting the change of time parameter $t \mapsto \frac{1-e^{-t}}{1+e^{-t}}$ (which has been used by many authors before).

Throughout the paper, we make use of the following notation. The function $\phi:\mathbb{R}^d \to \mathbb{R}$ will have action $x \mapsto \frac{x^2}{2}$. The standard Gaussian measure $\mu$ on $\mathbb{R}^d$ will thus be written with density $d\mu = (2\pi)^{-\frac{d}{2}} e^{-\phi(x)}dx$. The Lebesgue measure on $\mathbb{R}^d$ will be denoted by $\lambda$. As we only ever work over $\mathbb{R}^d$ with Borel $\sigma$-algebra, the measurable space over which we consider Lebesgue spaces will be dropped from the notation. For $\theta \in [0, \pi]$, we will write $\Sigma_\theta$ for the open sector $\{z \in \mathbb{C}\backslash\{0\}; |\arg(z)|<\theta\}$.

\section{R-Sectoriality of L}

To simplify things, for the rest of the article we will assume that $p \in (1, \infty)$ is fixed. Similarly, all concepts of boundedness and R-boundedness will be on either $L^p(\mu)$ or $L^p(\lambda)$ without explicit mention of the space, the measure being clear from context.

\begin{lemma}\label{lemmaKernel}
$M_t$ has the alternate form for $t>0$ and $x, y \in \mathbb{R}^d$,
\[M_t(x,y) =  \frac{1}{(2\pi)^\frac{d}{2}}\left(\frac{1}{1-e^{-2t}}\right)^\frac{d}{2}\exp\left(-s_t \left(\frac{x+y}{2\sqrt{2}}\right)^2 - \frac{1}{4s_t}\left(\frac{x-y}{\sqrt{2}}\right)^2   \right)\exp\left(\frac{1}{2}\left(\phi(x) - \phi(y)\right)\right),\]
where $s_t = \frac{1-e^{-t}}{1+e^{-t}}$.
\end{lemma}

\begin{proof}
We will rearrange the exponent from Equation (\ref{defnOU}) and show that it is equal to the exponent given above for all $x, y \in \mathbb{R}^d$ and $t>0$, as that is all that has changed between the two representations. For each $t>0, x, y \in \mathbb{R}^d$ we have

\begin{align*}
 -\frac{1}{2}\frac{|e^{-t}x - y|^2}{(1-e^{-2t})} &= -\frac{1}{2}\frac{|e^{-t}x - y|^2}{(1-e^{-2t})} - \frac{1}{4}(x^2-y^2)  + \frac{1}{4}(x^2-y^2) \\
 &= -\frac{1}{2}\frac{|e^{-t}x - y|^2}{(1-e^{-2t})} - \frac{1}{4}(x^2-y^2)  + \frac{1}{2}\left(\phi(x)-\phi(y)\right) \\
&= -\frac{1}{2(1-e^{-2t})}\left(|e^{-t}x - y|^2 + \frac{(1-e^{-2t})}{2}(x^2-y^2)\right)  + \frac{1}{2}\left(\phi(x)-\phi(y)\right) \\
&= -\frac{1}{2(1-e^{-2t})}\left(e^{-2t}x^2 - 2e^{-t}xy + y^2 + \frac{(1-e^{-2t})}{2}(x^2-y^2)\right)  + \frac{1}{2}\left(\phi(x)-\phi(y)\right) \\
&= -\frac{1}{2(1-e^{-2t})}\left(\frac{1}{2}\left(1+e^{-2t}\right)x^2 - 2e^{-t}xy + \frac{1}{2}\left(1+e^{-2t}\right)y^2\right)  + \frac{1}{2}\left(\phi(x)-\phi(y)\right) \\
&= -\frac{1}{8(1-e^{-2t})}\\
&\times \left(\left((1+e^{-t})^2 + (1-e^{-t})^2\right)x^2 + 2\left((1-e^{-t})^2 - (1+e^{-t})^2\right)xy + \left((1+e^{-t})^2 + (1-e^{-t})^2\right)y^2\right)  \\
&+ \frac{1}{2}\left(\phi(x)-\phi(y)\right) \\
&= -\frac{1}{8(1-e^{-2t})}\left( (1-e^{-t})^2(x+y)^2 + (1+e^{-t})^2(x-y)^2\right)  + \frac{1}{2}\left(\phi(x)-\phi(y)\right) \\
&= -\left(\frac{1-e^{-t}}{1+e^{-t}}\left(\frac{x+y}{2\sqrt{2}}\right)^2 + \frac{1}{4}\frac{1+e^{-t}}{1-e^{-t}}\left(\frac{x-y}{\sqrt{2}}\right)^2\right)  + \frac{1}{2}\left(\phi(x)-\phi(y)\right) \\
&= -\left(s_t\left(\frac{x+y}{2\sqrt{2}}\right)^2 + \frac{1}{4s_t}\left(\frac{x-y}{\sqrt{2}}\right)^2\right)  + \frac{1}{2}\left(\phi(x)-\phi(y)\right). \\
\end{align*}
\end{proof}

The next definition, albeit a simple one, forms the backbone of the rest of our arguments.

\begin{defn}\label{defnUp}
Define the (multiple of an) isometry $U_p: L^p(\mu) \to L^p(\lambda)$ by
\[U_pf(x) = f(x) \exp\left(-\frac{\phi(x)}{p}\right), \text{ for } x \in \mathbb{R}^d.\]
\end{defn}

As explained previously, we need only show that the Ornstein-Uhlenbeck semigroup has an analytic extension to a sector of the correct angle, and that it is R-bounded on each smaller sector. We will in fact show a lot more with no more effort. We shall work with the reparametrisation of the kernel of the semigroup in terms of $s_t$ from Lemma \ref{lemmaKernel}. The function $t \mapsto s_t$ is analytic and can clearly be analytically extended to the domain $\mathbb{C}\backslash i\pi(2\mathbb{Z}+1)$. We will consider the analytic extension $z \mapsto s_z$ on domains of the form
\begin{equation}\label{eqnE}
E := \left\{z \in \mathbb{C}; s_z \in \Sigma_{\frac{\pi}{2}-\theta_p}; z \not\in i\pi\mathbb{Z}\right\}
\end{equation}
where $\sin (\theta_p) = M_p := \left|1-\frac{2}{p}\right|$. We will show the Ornstein-Uhlenbeck semigroup extends to an analytic semigroup on the domain $E$. Moreover, we will simultaneously show that the Ornstein-Uhlenbeck semigroup is R-bounded on sets of the form
\begin{equation}\label{eqnEdash}
E_{\epsilon, \delta} = \left\{z \in \mathbb{C}; |\Re(s_z)|^2/|s_z|^2 = \cos^2(\arg(s_z))> M_p^2 + \epsilon; \text{dist}\left(z, i\pi(2\mathbb{Z}+1)\right)>\delta ; z \not\in 2i\pi\mathbb{Z}\right\}
\end{equation}
for all $\epsilon, \delta>0$. Note that, in terms of the reparametrisation $s_z$, these sets are open sectors of angle $\frac{\pi}{2}-\theta_p$ or less, with certain points removed. We claim that $\Sigma_{\frac{\pi}{2}-\theta_p} \subset E$, and that for all $\epsilon'>0$ there exists $\epsilon, \delta>0$ such that $\Sigma_{\frac{\pi}{2}-\theta_p-\epsilon'} \subset E_{\epsilon, \delta}$ (see \cite{vNPWeylGaussian} for details of this calculation). These results combined will imply that the maximal domain of analyticity of the Ornstein-Uhlenbeck semigroup contains the sector $\Sigma_{\frac{\pi}{2}-\theta_p}$, and that it is R-bounded on each smaller sector, which combined with the procedure outlined in the introduction will show at least that the Ornstein-Uhlenbeck operator is R-sectorial of the desired angle.

\begin{thm}\label{thmRSectorialGoodAngle}
For $p \in (1, \infty)$, the Ornstein-Uhlenbeck operator on $L^p(\mu)$ is R-sectorial of angle $\theta_p$, where $\sin (\theta_p) = M_p := \left|1-\frac{2}{p}\right|$.
\end{thm}

\begin{proof}

To determine (R-)boundedness of the analytic extension of $\exp(-tL)$ on $L^p(\mu)$ we conjugate by the (multiple of an) isometry $U_p: L^p(\mu) \to L^p(\lambda)$, and work with $U_p \exp(-tL) U_p^{-1}$ on $L^p(\lambda)$. As (multiples of) isometries preserve (R-)boundedness, $\exp(-tL)$ has an analytic extension to $z \in \mathbb{C}$ if and only if $U_p \exp(-tL) U_p^{-1}$ does, and both families of operators will be R-bounded on the same subdomains of the domain of analyticity. Using the integral kernel of Lemma \ref{lemmaKernel} and the explicit form of the isometry $U_p$ from Definition \ref{defnUp}, we find the integral representation for $f \in L^p(\lambda)$:
\[U_p \exp(-tL) U_p^{-1}f = \left( x \mapsto \int_{\mathbb{R}^d} k_t(x,y)f(y)dy\right),\]
with
\[k_t(x,y) = \frac{1}{(2\pi)^\frac{d}{2}}\left(\frac{1}{1-e^{-2t}}\right)^\frac{d}{2}\exp\left(-s_t \left(\frac{x+y}{2\sqrt{2}}\right)^2 - \frac{1}{4s_t}\left(\frac{x-y}{\sqrt{2}}\right)^2   \right)\exp\left(\left(\frac{1}{2}-\frac{1}{p}\right)\left(\phi(x) - \phi(y)\right)\right)\]
and $s_t = \frac{1-e^{-t}}{1+e^{-t}}$.
If $U_p \exp(-tL) U_p^{-1}$ were to have an analytic extension $U_p \exp(-zL) U_p^{-1}$ for $z$ in some domain containing $[0,\infty)$, uniqueness theory of analytic functions implies that $U_p \exp(-zL) U_p^{-1}$ would also have an integral representation, with kernel
\[k_z(x,y) = \frac{1}{(2\pi)^\frac{d}{2}}\left(\frac{1}{1-e^{-2z}}\right)^\frac{d}{2}\exp\left(-s_z \left(\frac{x+y}{2\sqrt{2}}\right)^2 - \frac{1}{4s_z}\left(\frac{x-y}{\sqrt{2}}\right)^2   \right)\exp\left(\left(\frac{1}{2}-\frac{1}{p}\right)\left(\phi(x) - \phi(y)\right)\right),\]
where $s_z = \frac{1-e^{-z}}{1+e^{-z}}$. To understand why this must be the case, we can act $U_p \exp(-tL) U_p^{-1}$ on some $f \in L^p(\mu)$, and then pair with some $g \in \left(L^p(\mu)\right)^* = L^{p'}(\mu)$ to obtain a function $\mathbb{R}^+ \to \mathbb{C}$, $t \mapsto \left\langle U_p \exp(-tL) U_p^{-1}f, g\right\rangle$. This function will have an analytic extension to the set of $z$ for which the operator with integral kernel $k_z(x,y)$ is bounded on $L^p(\mu)$, and standard uniqueness results for $\mathbb{C}$-valued analytic functions implies that the analytic extension will be given by the operator with integral kernel $k_z(x,y)$ applied to $f$ and paired with $g$. Thus $U_p \exp(-tL) U_p^{-1}$ would have as weak-analytic extension the operator with integral kernel $k_z(x,y)$, to the set of $z$ for which this is  bounded on $L^p(\mu)$. By the equivalence of strong-analytic and weak-analytic Banach space valued functions (see, for example, Chapter VII \S 3, Exercise 4 of \cite{Conway}), the claim follows. (There is a slight notational issue here, in that the definition of an analytic semigroup on a Banach space $X$ is only ever analytic in the strong operator topology, such that the functions $z \mapsto \exp(-zL)f$ are $X$-valued norm-analytic functions, for each $f \in X$).

We will now work on bounding $k_z(x,y)$. We start by assuming that $z \in E$ (see Equation (\ref{eqnE})). Note that this implies $Re(s_z) > 0$ and $1-e^{-2z} \neq 0$. Then we have:
\begin{align*}
|k_z(x,y)| &\leq \frac{1}{(2\pi)^\frac{d}{2}}\left|\frac{1}{1-e^{-2z}}\right|^\frac{d}{2}\exp\left(-\Re (s_z) \left(\frac{x+y}{2\sqrt{2}}\right)^2 - \frac{1}{4} \Re\left(\frac{1}{s_z}\right)\left(\frac{x-y}{\sqrt{2}}\right)^2   \right)\exp\left(\left(\frac{1}{2}-\frac{1}{p}\right)\left(\phi(x) - \phi(y)\right)\right) \\
&\leq \frac{1}{(2\pi)^\frac{d}{2}}\left|\frac{1}{1-e^{-2z}}\right|^\frac{d}{2}\exp\left(-\Re (s_z) \left(\frac{x+y}{2\sqrt{2}}\right)^2 + M_p \frac{1}{4}\left|x^2 - y^2\right| - \frac{1}{4} \Re\left(\frac{1}{s_z}\right)\left(\frac{x-y}{\sqrt{2}}\right)^2 \right) \\
&= \frac{1}{(2\pi)^\frac{d}{2}}\left|\frac{1}{1-e^{-2z}}\right|^\frac{d}{2}\exp\left(-\Re (s_z) \left(\frac{x+y}{2\sqrt{2}}\right)^2 + M_p \left|\frac{x+y}{2\sqrt{2}}\right| \left|\frac{x-y}{\sqrt{2}}\right| - \frac{1}{4} \Re\left(\frac{1}{s_z}\right)\left(\frac{x-y}{\sqrt{2}}\right)^2 \right) \\
\end{align*}
For notational simplicity, let $u = \left|\frac{x+y}{2\sqrt{2}}\right|$ and $k = \left|\frac{x-y}{\sqrt{2}}\right|$. Then rewriting in terms of $u$ and $k$ and completing the square in $u$ gives
\begin{align*}
|k_z(x,y)| &\leq \frac{1}{(2\pi)^\frac{d}{2}}\left|\frac{1}{1-e^{-2z}}\right|^\frac{d}{2}\exp\left(-\Re (s_z) u^2 + M_p uk - \frac{1}{4} \Re\left(\frac{1}{s_z}\right)k^2 \right) \\
&= \frac{1}{(2\pi)^\frac{d}{2}}\left|\frac{1}{1-e^{-2z}}\right|^\frac{d}{2}\exp\left(-\left(\sqrt{\Re (s_z)}u - \frac{M_p}{2\sqrt{\Re (s_z)}} k\right)^2 - \frac{1}{4} \left(\Re\left(\frac{1}{s_z}\right) - \frac{M_p^2}{\Re (s_z)}\right)k^2 \right). \\
\end{align*}
So
\begin{align*}
|k_z(x,y)| &\leq \frac{1}{(2\pi)^\frac{d}{2}}\left|\frac{1}{1-e^{-2z}}\right|^\frac{d}{2}\exp\left(- \frac{1}{4} \left(\Re\left(\frac{1}{s_z}\right) - \frac{M_p^2}{\Re (s_z)}\right)k^2 \right) \\
&= \frac{1}{(2\pi)^\frac{d}{2}}\left|\frac{1}{1-e^{-2z}}\right|^\frac{d}{2}\exp\left(- \frac{1}{4} \left(\Re\left(\frac{1}{s_z}\right) - \frac{M_p^2}{\Re (s_z)}\right)\left(\frac{x-y}{\sqrt{2}}\right)^2 \right). \\
&= \frac{1}{(2\pi)^\frac{d}{2}}\left|\frac{1}{1-e^{-2z}}\right|^\frac{d}{2}\exp\left(- \frac{1}{8} \left(\Re\left(\frac{1}{s_z}\right) - \frac{M_p^2}{\Re (s_z)}\right)\left(x-y\right)^2 \right). \\
\end{align*}
For each $z \in E$, let $g_z:\mathbb{R}^d \to \mathbb{R}$ be the function
\[x \mapsto \frac{1}{(2\pi)^\frac{d}{2}}\left|\frac{1}{1-e^{-2z}}\right|^\frac{d}{2}\exp\left(- \frac{1}{8} \left(\Re\left(\frac{1}{s_z}\right) - \frac{M_p^2}{\Re (s_z)}\right)x^2 \right).\]
Then we have that for all $z \in E$, $f \in L^p(\lambda)$ and a.e. $x \in \mathbb{R}^d$
\[\left| \left(U_p \exp(-zL) U_p^{-1}f\right)(x)\right| \leq (g_z * |f|)(x).\] 
Therefore, provided the family of convolution operators $f \in L^p(\lambda) \mapsto g_z*f$ is (R-)bounded for $z$ in (a subset of) $E$, we will have proven, by domination and isometry, that $\exp(-zL)$ is (R-)bounded on (the same subset of) $E$ (to see that domination implies R-boundedness, see Proposition 8.1.10 of \cite{AnalysisInBanachSpacesV2}, and note that in the proof of said proposition the fixed positive operator can be replaced by an R-bounded family of positive operators). For $z \in E$, we find
\[ \Re\left(\frac{1}{s_z}\right) - \frac{M_p^2}{\Re (s_z)} = \frac{\Re(s_z)}{|s_z|^2} - \frac{M_p^2}{\Re (s_z)} 
= \frac{1}{\Re (s_z)}\left(\frac{|\Re(s_z)|^2}{|s_z|^2} - M_p^2\right) 
>0,\]
since $\Re (s_z)>0$ and $|\Re(s_z)|^2/|s_z|^2 = \cos^2(\arg(s_z)) > M_p^2$ by definition of $E$ (since $\cos\left(\frac{\pi}{2} - \theta_p\right) = \sin\left(\theta_p\right) = M_p$). So for $z \in E$, $g_z \in L^1(\lambda)$ and so by Young's convolution inequality, convolution by $g_z$ is a bounded operator on $L^p(\lambda)$ with operator norm at most $||g_z||_{L^1(\lambda)}$.
Now we will focus on sets of the form $E_{\epsilon, \delta}$ for some fixed $\epsilon, \delta >0$ (see Equation (\ref{eqnEdash})). We will show that
 \begin{equation*}
 \sup_{z \in E_{\epsilon, \delta}} \int_{\mathbb{R}^d} \sup_{|y|>|x|} |g_z(y)| dx < \infty,
 \end{equation*}
from which we can apply Proposition 8.2.3 of \cite{AnalysisInBanachSpacesV2} to find that the family of convolution operators $\{g_z*\}_{z \in E_{\epsilon, \delta}}$ is R-bounded on $L^p(\lambda)$. Noting that each $g_z$ is radially decaying and positive, the quantity to bound is
\begin{align*}
\sup_{z \in E_{\epsilon, \delta}} \int_{\mathbb{R}^d} \sup_{|y|>|x|} |g_z(y)| dx &= \sup_{z \in E_{\epsilon, \delta}} \int_{\mathbb{R}^d} g_z(x) dx \\
&= \sup_{z \in E_{\epsilon, \delta}} \int_{\mathbb{R}^d} \frac{1}{(2\pi)^\frac{d}{2}}\left|\frac{1}{1-e^{-2z}}\right|^\frac{d}{2}\exp\left(- \frac{1}{8} \left(\Re\left(\frac{1}{s_z}\right) - \frac{M_p^2}{\Re (s_z)}\right)x^2 \right) dx \\
&= \sup_{z \in E_{\epsilon, \delta}}  \frac{1}{(2)^\frac{d}{2}}\left|\frac{1}{1-e^{-2z}}\right|^\frac{d}{2} \left(\frac{1}{8} \left(\Re\left(\frac{1}{s_z}\right) - \frac{M_p^2}{\Re (s_z)}\right)\right)^{-\frac{d}{2}} \\
&\leq \sup_{z \in E_{\epsilon, \delta}}  2^d \left|\frac{1}{1-e^{-2z}}\right|^\frac{d}{2} \left(\frac{\epsilon}{\Re(s_z)}\right)^{-\frac{d}{2}} \\
&= \sup_{z \in E_{\epsilon, \delta}}  \epsilon^{-\frac{d}{2}}2^d \left(\left|\frac{\Re(s_z)}{1-e^{-2z}}\right| \right)^{\frac{d}{2}} \\
&\leq \sup_{z \in E_{\epsilon, \delta}}  \epsilon^{-\frac{d}{2}}2^d \left(\frac{|s_z|}{|1-e^{-z}||1+e^{-z}|} \right)^{\frac{d}{2}} \\
&= \sup_{z \in E_{\epsilon, \delta}}  \epsilon^{-\frac{d}{2}}2^d \left(\frac{\left|\frac{1-e^{-z}}{1+e^{-z}}\right|}{|1-e^{-z}||1+e^{-z}|} \right)^{\frac{d}{2}} \\
&= \sup_{z \in E_{\epsilon, \delta}}  \epsilon^{-\frac{d}{2}}2^d \left(\frac{1}{|1+e^{-z}|} \right)^d \\
&<\infty
\end{align*}
since $z$ is bounded away from $(2\mathbb{Z}+1)i\pi$. So the family of convolution operators $\{g_z*\}_{z \in E_{\epsilon, \delta}}$ is R-bounded.
By pointwise domination, $U_p \exp(-zL) U_p^{-1}$ is bounded for $z \in E$, and is R-bounded on subsets $E_{\epsilon, \delta} \subset E$ of the form (\ref{eqnEdash}). Hence by isometric equivalence, $\exp(-zL)$ shares the same properties. Hence the claim follows from the discussion preceding this proof.
\end{proof}


\bibliography{mybib}
\bibliographystyle{plain}

\end{document}